\documentclass[12pt]{amsart}
\usepackage{amsmath}
\usepackage{eucal}
\usepackage{amssymb}
\usepackage[all]{xy}
\usepackage{longtable}
\usepackage[osf,sc]{mathpazo}
\usepackage[titletoc]{appendix}
\usepackage{color}
\usepackage{multirow,bigdelim}
\usepackage{stmaryrd}
\usepackage{verbatim}
\usepackage{mathrsfs}
\usepackage{tikz-cd}

\allowdisplaybreaks

\newtheorem{theorem}[equation]{Theorem}
\newtheorem{lemma}[equation]{Lemma}

\newtheorem{corollary}[equation]{Corollary}

\newtheorem{definition}[equation]{Definition}

\newtheorem{example}[equation]{Example}

\theoremstyle{remark}
\newtheorem{remark}[equation]{Remark}

\numberwithin{equation}{subsection}

\allowdisplaybreaks[1]

\DeclareMathAlphabet{\matheur}{U}{eur}{m}{n}

 \DeclareMathOperator{\Lie}{Lie}
 
\DeclareMathOperator{\Mat}{Mat} 
\DeclareMathOperator{\End}{End}

 \DeclareMathOperator{\im}{Im}
 
\DeclareMathOperator{\Ext}{Ext}  
\DeclareMathOperator{\rank}{rank}

\newcommand{\ok}{\overline{k}}

\newcommand{\tr}{\mathrm{tr}}

\begin{document}

\title[]{{\large{L\MakeLowercase{inear equations on }D\MakeLowercase{rinfeld modules}}}}

\author{Yen-Tsung Chen}
\address{Department of Mathematics, National Tsing Hua University, Hsinchu City 30042, Taiwan
  R.O.C.}

\email{ytchen.math@gmail.com}

\thanks{The author was partially supported by Prof. C.-Y. Chang's MOST Grant 107-2628-M-007-002-MY4}

\subjclass[2010]{Primary 11G09, 11R58}

\date{\today}

\begin{abstract} 
    Let $L$ be a finite extension of the rational function field over a finite field $\mathbb{F}_q$ and $E$ be a Drinfeld module defined over $L$. Given finitely many elements in $E(L)$, this paper aims to prove that linear relations among these points can be characterized by solutions of an explicitly constructed system of homogeneous linear equations over $\mathbb{F}_q[t]$. As a consequence, we show that there is an explicit upper bound for the size of the generators of linear relations among these points. This result can be regarded as an analogue of a theorem of Masser for finitely many $K$-rational points on an elliptic curve defined over a number field $K$.
\end{abstract}

\keywords{}

\maketitle

\tableofcontents

\section{Introduction}
\subsection{Motivation}
    Let $K$ be a number field and $E$ be an elliptic curve defined over $K$.
    Given finitely many non-zero distinct $K$-rational points $P_1,\dots,P_\ell\in E(K)$, how can we decide if they are linearly dependent in the sense that there exist integers $a_1,\dots,a_\ell$, not all zero, such that $a_1P_1+\cdots+a_\ell P_\ell=0$ on $E$? Furthermore, one may naturally ask whether we can determine all the linear relations among these points.
    This question was answered by Masser in \cite{Mas88}, using techniques from the geometry of numbers. 
    To be more precise, we embed our elliptic curves in the projective space $\mathbb{P}^2$ by choosing a Weierstrass model 
    \[
        Y^2Z=4X^3-g_2XZ^2-g_3Z^3.
    \]
    Thus, the invariants $g_2,g_3$ lie in the number field $K$. Let $\Hat{h}$ be the Neron-Tate height on $E(K)$. We set $\eta_1:=\inf_{P\in E(K)\setminus E(K)_\mathrm{tor}}\{\Hat{h}(P)\}\neq 0$ and $\omega:=|E(K)_{\mathrm{tor}}|$. For $\mathbf{x}:=(x_1,\dots,x_\ell)\in\mathbb{Z}^\ell$, we define $|\mathbf{x}|:=\max_{i=1}^\ell\{|x_i|\}$. Then Masser's theorem is stated as follows.
    
    \begin{theorem}[{\cite[Thm.~E]{Mas88}}]\label{Thm:Masser's_Thm}
        Let $K/\mathbb{Q}$ be a finite extension. Let $E$ be an elliptic curve defined over $K$ and $P_1,\dots,P_\ell\in E(K)$ be distinct non-zero points with Neron-Tate heights at most $\eta_2\geq\eta_1$. If we set
        $$G:=\{(a_1,\dots,a_\ell)\in\mathbb{Z}^\ell\mid\sum_{i=1}^\ell a_iP_i=0\}\subset\mathbb{Z}^\ell,$$
        then there exists a set $\{\mathbf{m}_1,\dots,\mathbf{m}_\nu\}\subset G$ such that $G=\mathrm{Span}_{\mathbb{Z}}\{\mathbf{m}_1,\dots,\mathbf{m}_\nu\}$
        with bounded size
        $$|\mathbf{m}_i|\leq\ell^{\ell-1}\omega(\eta_2/\eta_1)^{(\ell-1)/2}$$
        for each $1\leq i\leq\nu$.
    \end{theorem}
    
    In fact, he generalized these results and methods to connected commutative algebraic groups defined over a number field $K$. We refer the reader to \cite{Mas88} for details. This paper aims to study an analogue of Masser's result for Drinfeld modules.
\subsection{The main result}
    Let $\mathbb{F}_q$ be a fixed finite field with $q$ elements, for $q$ a power of a prime number $p$. 
    Let $\mathbb{P}^1$ be the projective line defined over $\mathbb{F}_q$ with a fixed point at infinity $\infty\in\mathbb{P}^1(\mathbb{F}_q)$. 
    Let $A$ be the ring of algebraic functions regular away from $\infty$ and $k$ be its fraction field. Let $k_\infty$ be the completion of $k$ at $\infty$ and $\mathbb{C}_\infty$ be the completion of a fixed algebraic closure of $k_\infty$. Let $\theta$ be a variable. We identify $A$ with the polynomial ring $\mathbb{F}_q[\theta]$ and $k$ with the rational function field $\mathbb{F}_q(\theta)$. Throughout this paper, we fix $\ok$ to be an algebraic closure of $k$ inside $\mathbb{C}_\infty$.
    
    
    Let $R$ be an $A$-algebra and $\tau:=(x\mapsto x^q):R\to R$ be the Frobenius $q$-th power operator. We define $R[\tau]$ to be the twisted polynomial ring in $\tau$ over $R$ subject to the relation $\tau\alpha=\alpha^q\tau$ for $\alpha\in R$. Let $t$ be a new variable and let $\mathbb{F}_q[t]$ be the polynomial ring in $t$ over $\mathbb{F}_q$. A Drinfeld $\mathbb{F}_q[t]$-module over $R$ is a pair $E=(\mathbb{G}_{a/R},\phi)$ where $\mathbb{G}_{a/R}$ is the additive group scheme defined over $R$ and $\phi$ is an $\mathbb{F}_q$-algebra homomorphism $\phi:\mathbb{F}_q[t]\to R[\tau]$ so that $\im(\phi)\not\subset R$, and  $\phi_t=\theta+\kappa_{1}\tau+\cdots+\kappa_{r}\tau^r\in R[\tau]$ with $\kappa_r\neq 0$. The rank of $E$ is defined to be $r$. Let $E=(\mathbb{G}_{a/R},\phi)$ and $E'=(\mathbb{G}_{a/R},\phi)$ be two Drinfeld $\mathbb{F}_q[t]$-modules defined over $R$. An element $u\in R[\tau]$ is called a $R$-morphism from $E$ to $E'$ if $u\phi_t=\phi'_tu$, and it is called an $R$-isomorphism if there is another $R$-morphism $u'$ from $E'$ to $E$ such that $uu'=u'u=1$. By abuse of notation, if $u$ is a $R$-isomorphism from $E$ to $E'$, we will denote by $u(E):=E'$.
    
    Let $L/k$ be a finite extension and $E=(\mathbb{G}_{a/L},\phi)$ be a Drinfeld $\mathbb{F}_q[t]$-module defined over $L$. We denoted by $E(L)$ the $\mathbb{F}_q[t]$-module whose underlying space is the additive group $L$, and the $\mathbb{F}_q[t]$-module structure arises from $\phi$. It is known in \cite{Den92a} that $E(L)$ is never finitely generated and thus the direct analogue of the Mordell-Weil theorem fails. Later on, Poonen proved in \cite{Poo95} that $E(L)$ is isomorphic to the direct sum of a free $\mathbb{F}_q[t]$-module of countably infinite rank $\aleph_0$ with a finite $\mathbb{F}_q[t]$-torsion submodule $E(L)_\mathrm{tor}$. An analogue of Masser's theorem for Drinfeld modules was considered by Denis in \cite[Appendix~2]{Den92a}, using his theory of the canonical height on Drinfeld modules \cite{Den92b}. His approach was inspired by Masseer's study on abelian varieties \cite{Mas81}. In fact, Denis's result works for a more general setting of $t$-modules.
    
    In the present paper, we adopt a different method to study an analogue of Masser's theorem for Drinfeld modules. Our strategy relies on the theory of Anderson dual $t$-motive and the investigation of Frobenius difference equations. Furthermore, given finitely many algebraic points in $E(L)$, we prove that all the $\mathbb{F}_q[t]$-linear relations among these points can be captured by the solutions of an explicitly constructed homogeneous linear system over $\mathbb{F}_q[t]$. We shall emphasize that the former $\mathbb{F}_q[t]$-action on algebraic points in $E(L)$ arises from $\phi$, while the later $\mathbb{F}_q[t]$-action in the homogeneous linear system is the usual scalar multiplication. Based on this characterization, we established an analogue of Theorem~\ref{Thm:Masser's_Thm} for Drinfeld modules, which is different from the result of Denis in \cite[Appendix~2]{Den92a}.
    
    To state our result, we introduce a divisor $D_E$ over $L$ for a given Drinfeld $\mathbb{F}_q[t]$-module $E=(\mathbb{G}_{a/L},\phi)$ with $\phi_t=\theta+\kappa_1\tau+\cdots+\kappa_r\tau^r\in L[\tau]$ (see \eqref{Eq:slope_divisor} for the detailed description). Note that the divisor $D_E$ is different from the one defined in \cite[Sec.~2]{Tag93}, and the definition of $D_E$ is related to the slope of the Newton polygon of the polynomial 
    \begin{equation}\label{Eq:Newton_Polygon_of_Drinfeld}
        \mathcal{P}_E(x):=(t-\theta)x+\kappa_1x^q+\cdots+\kappa_rx^{q^r}\in L[t][x].
    \end{equation}
    Moreover, $\deg(D_E)$ is invariant under $L$-isomorphisms in the following sense: Let $u$ be an $L$-isomorphism of $E$. Then $\deg(D_E)=\deg(D_{u(E)})$. Thus, $\deg(D_E)$ is well-defined on the $L$-isomorphism classes of Drinfeld $\mathbb{F}_q[t]$-modules defined over $L$ (see Lemma~\ref{Lem:Invariant}).
    
    For $B=(B_{ij})\in\Mat_{m\times n}(\mathbb{F}_q[t])$, we set $\deg_t(B):=\max_{i,j}\{\deg_t(B_{ij})\}$. Let $n\in\mathbb{Z}_{>0}$ and $\mathbf{x}=(x_1,\dots,x_n)\in L^n\setminus\{0\}$. We have an associated divisor $\mathrm{div}(\mathbf{x})$ over $L$ (see Definition~\ref{Def:Height}). The height of $\mathbf{x}$ is given by $h(\mathbf{x}):=-\deg(\mathrm{div}(\mathbf{x}))/[L:k]$. Let $M_L$ be the set of places of $L$. For each $v\in M_L$ we fix a uniformizer $\varpi_v$ of $v$ in $L$ and an additive valuation $\mathrm{ord}_v(\cdot)$ such that $\mathrm{ord}_v(\varpi_v)=1$.
    For two divisors $D_1:=\sum_{v\in M_L}a_v\cdot v$ and $D_2:=\sum_{v\in M_L}b_v\cdot v$, we define the divisor
    \begin{equation}\label{Eq:max_divisors}
        D_1\vee D_2:=\sum_{v\in M_L}\max\{a_v,b_v\}\cdot v.
    \end{equation}
    For a divisor $D$ over $L$, we set 
    \begin{equation}\label{Eq:Riemann_Roch_Space}
        \mathcal{L}(D):=\{f\in L^\times\mid \mathrm{div}(f)+D\geq 0\}\cup\{0\}.
    \end{equation}
    
    
    Our first result, restated as Theorem~\ref{Thm:Reduction_2}, asserts that all $\mathbb{F}_q[t]$-linear relations among finitely elements in $E(L)$ of a given Drinfeld $\mathbb{F}_q[t]$-module $E$ defined over $L$ can be recovered from a system of linear equations over $\mathbb{F}_q[t]$.
    \begin{theorem}\label{Thm:Intro_Thm1}
        Let $L/k$ be a finite extension. Let $E=(\mathbb{G}_{a/L},\phi)$ be a Drinfeld $\mathbb{F}_q[t]$-module defined over $L$ and $P_1,\dots,P_\ell\in E(L)$ be distinct non-zero elements in $L$. Let
        $$D:=\left(-\mathrm{div}(P_1,\dots,P_\ell)\right)\vee D_E$$
        and $d:=\dim_{\mathbb{F}_q}\mathcal{L}(D)$.
        Then there exists an explicitly constructed matrix $B$ with entries in $\mathbb{F}_q[t]$ and $\deg_t(B)\leq 1$ such that the canonical projection
        \begin{align*}
            \pi:\{\mathbf{x}\in\mathbb{F}_q[t]^{(d+\ell)}\mid B\mathbf{x}^\tr=0\}&\twoheadrightarrow \{(a_1,\dots,a_\ell)\in\mathbb{F}_q[t]^\ell\mid \sum_{i=1}^\ell\phi_{a_i}(P_i)=0\}\\
            (g_1,\dots,g_d,a_1,\dots,a_\ell)&\mapsto (a_1,\dots,a_\ell)
        \end{align*}
        is a well-defined surjective $\mathbb{F}_q[t]$-module homomorphism.
    \end{theorem}
    
    We provide an example of detailed calculations of the matrix $B$ in Example~\ref{Ex:Carlitz}. As a consequence of Theorem~\ref{Thm:Intro_Thm1}, our second result, can be viewed as an analogue of Masser's theorem for Drinfeld $\mathbb{F}_q[t]$-modules.
    \begin{theorem}\label{Thm:Main_Thm}
        Let $L/k$ be a finite extension. Let $E=(\mathbb{G}_{a/L},\phi)$ be a Drinfeld $\mathbb{F}_q[t]$-module defined over $L$ and $P_1,\dots,P_\ell\in E(L)$ be distinct non-zero elements in $E(L)$. If we set
        $$G:=\{(a_1,\dots,a_\ell)\in\mathbb{F}_q[t]^\ell\mid \sum_{i=1}^\ell\phi_{a_i}(P_i)=0\}\subset\mathbb{F}_q[t]^\ell,$$
        then there exists a set $\{\mathbf{m}_1,\dots,\mathbf{m}_\nu\}\subset G$ with bounded degree $\deg_t(\mathbf{m}_i)<d+\ell$ such that $$G=\mathrm{Span}_{\mathbb{F}_q[t]}\{\mathbf{m}_1,\dots,\mathbf{m}_\nu\}.$$
        In addition, we have $\deg(D)\geq [L:k]\cdot h(P_1,\dots,P_\ell)$ and the equality holds if the divisor $-\left(\mathrm{div}(P_1,\dots,P_\ell)+D_E\right)$ is effective. Moreover, $\deg(D)$ is invariant under $L$-isomorphisms of Drinfeld $\mathbb{F}_q[t]$-modules in the following sense: Let $u$ be an $L$-isomorphism of $E$. If we set
        $$u(D)=\left(-\mathrm{div}(u(P_1),\dots,u(P_\ell))\right)\vee D_{u(E)},$$
        then we have $\deg(D)=\deg(u(D))$.
    \end{theorem}

    At the writing of this paper, it is not clear for the author whether this upper bound is optimal or not (see Remark.~\ref{Rmk:Optimality}). The author hope to tackle this question in a near future.
    
    \begin{remark}\label{Rmk:Endomorphism_Relations}
        In this remark, we explain how to deal with linear relations among algebraic points on Drinfeld $\mathbb{F}_q[t]$-modules with endomorphism ring larger than $\mathbb{F}_q[t]$. Let $L/k$ be a finite extension. Let $E=(\mathbb{G}_{a/L},\phi)$ be a Drinfeld $\mathbb{F}_q[t]$-module defined over $L$ and $P_1,\dots,P_\ell\in E(L)$ be distinct non-zero elements in $E(L)$. Let 
        \[
            \End(E):=\{u\in\ok[\tau]\mid u\phi_t=\phi_tu\}.
        \]
        Then the $\mathbb{F}_q[t]$-algebra structure on $\End(E)$ comes from
        \begin{align*}
            \phi:\mathbb{F}_q[t]&\to\End(E)\\
            a&\mapsto\phi_a.
        \end{align*}
        It is known that $\End(E)$ is an $\mathbb{F}_q[t]$-order and in particular it is free of finite rank $s$ over $\mathbb{F}_q[t]$ with $\mathbb{F}_q[t]$-basis $\{u_1,\dots,u_s\}$. One can see easily that $\End(E)$-linear relations among $P_1,\dots,P_\ell$ are equivalent to $\mathbb{F}_q[t]$-linear relations among $\{u_i(P_j)\mid1\leq i\leq s,~1\leq j\leq\ell\}$. In other words, Theorem~\ref{Thm:Intro_Thm1} and Theorem~\ref{Thm:Main_Thm} can be applied to determine $\End(E)$-linear relations among these points. 
    \end{remark}
    
    Let $K_E:=\End(E)\otimes_{\mathbb{F}_q[t]}\mathbb{F}_q(t)$. Then $K_E/\mathbb{F}_q(t)$ is a field extension of degree $s$. Note that the map
    \begin{align*}
        \partial:\End(E)&\to\mathbb{C}_\infty\\
        u:=\sum_{i=0}^mc_i\tau^i&\mapsto\partial u:=c_0
    \end{align*}
    extends to $K_E$ naturally by setting $\partial(\sum_{i=1}^mu_i\otimes \alpha_i):=\sum_{i=1}^m\partial(u_i)\cdot\alpha_i\mid_{t=\theta}$. Thus, $K_E$ can be viewed as a subfield of $\mathbb{C}_\infty$ via $\partial$. We denote by $\mathcal{K}_E:=\partial(K_E)\subset\mathbb{C}_\infty$. In analogy with classical transcendence results, Yu \cite{Yu97} established the analogue of Baker’s celebrated theorem on the linear independence of logarithms for Drinfeld logarithms at algebraic points. Moreover, Chang and Papanikolas proved in \cite[Thm.~1.1.1]{CP12} that if $E$ is a Drinfeld $\mathbb{F}_q[t]$-module defined over $\overline{k}$, then all algebraic relations among the logarithm of $E$ at algebraic points are those coming from  their $\mathcal{K}_E$-linear relations. Combining their result with Theorem~\ref{Thm:Main_Thm} and Remark~\ref{Rmk:Endomorphism_Relations}, one deduces a practical way to examine whether logarithms of Drinfeld $\mathbb{F}_q[t]$-modules at algebraic points are algebraically independent. We will give the precise statement as well as an example in Section~4.2.

\subsection{Strategy and organization}
    The key ingredients of our approach can be divided into two parts. The first part uses the theory of Anderson dual $t$-motives. More precisely, we identify algebraic points in a given Drinfeld $\mathbb{F}_q[t]$-module with elements in the quotient of its associated Anderson dual $t$-motive using a specific $t$-frame. Then we prove that finding $\mathbb{F}_q[t]$-linear relations among these algebraic points is equivalent to solving a particular Frobenius difference equation. The details of this part is given in Lemma~\ref{Lem:Reduction_1}.
    
    In the second part, we analyze the solution space of the Frobenius difference equation in question. In particular, we show that all the possible solutions can be described by polynomials in the variable $t$ with coefficients in a finite dimensional $\mathbb{F}_q$-vector space $\mathcal{L}(D)$, where $D$ is the divisor given in Theorem~$\ref{Thm:Intro_Thm1}$. In other words, we prove that the solution space is free of finite rank over $\mathbb{F}_q[t]$. Thus it can be characterised by a system of linear equations over $\mathbb{F}_q[t]$, and Theorem~\ref{Thm:Intro_Thm1} follows immediately. To prove Theorem~\ref{Thm:Main_Thm}, we apply a version of Siegel's lemma by Thunder \cite{Thu95} to deduce the desired upper bound. Note that Thunder's result can be seen as a function field analogue of the result by Bombieri and Vaaler \cite{BV83}.
    
    As another application of Theorem~\ref{Thm:Intro_Thm1}, we adopt a different analysis on the Frobenius difference equation obtained in the first step. It leads to a sufficient condition of algebraic points on a given Drinfeld module being $\mathbb{F}_q[t]$-linearly-linearly independent. Let $L/k$ be a finite extension. Let $S\subset M_L$ be a finite collection of places in $L$ which includes all places lying above $\infty\in M_k$. Let 
    $$\mathcal{O}_S:=\{\alpha\in L\mid \mathrm{ord}_v(\alpha)\geq 0~\mbox{for all}~v\not\in S\}\subset L$$
    be the set of $S$-integers in $L$. Then we have the following result, which will be stated as Theorem~\ref{Thm:Linearly_Independence} later.
    
    \begin{theorem}\label{Thm:Intro_Linearly_Independence}
        Let $E=(\mathbb{G}_{a/L},\phi)$ be a Drinfeld module $\mathbb{F}_q[t]$-module defined over $L$ with $\phi_t=\theta+\kappa_1\tau+\cdots+\kappa_r\tau^r\in\mathcal{O}_S[\tau]$ and $\kappa_r\in L^\times\cap\overline{\mathbb{F}}_q^\times$. Let $P_1,\dots,P_\ell\in E(L)$. Suppose that
        \begin{enumerate}
            \item $P_1,\dots,P_\ell$ are linearly independent over $\mathbb{F}_q$.
            \item $\mathrm{ord}_v(P_i)>0$ for all $1\leq i\leq\ell$, if $v\in S$.
            \item $\mathrm{ord}_v(P_i)\geq 1-q^r$ for all $1\leq i\leq\ell$, if $v\not\in S$.
        \end{enumerate}
        Then
        \[
            \rank_{\mathbb{F}_q[t]}\mathrm{Span}_{\mathbb{F}_q[t]}\{P_1,\dots,P_\ell\}=\ell.
        \]
        In particular, we have
        \[
            \rank_{\mathbb{F}_q[t]}E(L)=\aleph_0.
        \]
    \end{theorem}
    
    \begin{remark}
        Note that this result gives another proof of a part of \cite[Thm.~1]{Poo95} for a family of Drinfeld $\mathbb{F}_q[t]$-modules without proving the tameness of $E(L)$. An example of Theorem~\ref{Thm:Intro_Linearly_Independence} is given in Example~\ref{Ex:Linearly_Independent}.
    \end{remark}

    The organization of this paper is given as follows. In Section~2, we first recall a variant of Siegel's lemma for function fields \cite{Thu95} which can be regarded as an analogue of Bombieri and Vaaler's theorem \cite{BV83}. Then we follow the exposition of \cite{NP21} closely to review the theory of Anderson dual $t$-motives. In Section~3, we first show that to find linear relations among algebraic points on a given Drinfeld $\mathbb{F}_q[t]$-module is equivalent to find the solutions of a specific Frobenius difference equation. We mention that some technical arguments in the proof of this part are rooted in \cite{Cha16,KL16,Ho20}. Then we analyze the solutions of the Frobenius difference equation in question and we successfully construct an $\mathbb{F}_q[t]$-linear system whose solutions completely recover the solutions of the desired Frobenius difference equation. Finally we apply a Siegel's lemma for function fields to obtain an upper bound of the height of some $\mathbb{F}_q[t]$-linearly independent solutions of the $\mathbb{F}_q[t]$-linear system in question. In Section~4, we give two applications of our results. The first one is a sufficient condition of the linear independence of algebraic points on a specific family of Drinfeld $\mathbb{F}_q[t]$-modules. The second one is a practical way for determining algebraic relations among Drinfeld logarithms at algebraic points.
    
    
\section{Preliminaries}
\subsection{Table of Symbols}
    \begin{itemize}
		\setlength{\leftskip}{0.8cm}
        \setlength{\baselineskip}{18pt}
		\item[$\mathbb{F}_q$] :=  A fixed finite field with $q$ elements, for $q$ a power of a prime number $p$.
		\item[$\infty$] :=  A fixed closed point on the projective line $\mathbb{P}^1(\mathbb{F}_q)$.
		\item[$A$] :=  $\mathbb{F}_q[\theta]$, the rational functions of $\mathbb{P}^1$ which are regular away from $\infty$.
		\item[$k$] :=  $\mathbb{F}_q(\theta)$, the function field of $\mathbb{P}^1$.
		\item[$k_\infty$] :=  The completion of $k$ at the place $\infty$.
		\item[$\mathbb{C}_\infty$] :=  The completion of a fixed algebraic closure of $k_\infty$.
		\item[$\overline{k}$] := A fixed algebraic closure of $k$ inside $\mathbb{C}_\infty$.
		\item[$M_L$] := The set of all places of $L$ whenever $L/k$ is a finite extension of fields.
	\end{itemize}
	
	Recall that for each $v\in M_L$ we fix a uniformizer $\varpi_v$ of $v$ in $L$ and an additive valuation $\mathrm{ord}_v(\cdot)$ such that $\mathrm{ord}_v(\varpi_v)=1$. We extend $\mathrm{ord}_v(\cdot)$ to $\Mat_{m\times n}(L)$ by setting $\mathrm{ord}_v(B):=\min_{i,j}\{\mathrm{ord}_v(B_{ij})\}$ for each $B=(B_{ij})\in\Mat_{m\times n}(L)$. Recall that we fix the additive valuation $\deg_t(\cdot)$ on $\mathbb{F}_q[t]$ such that $\deg_t(t)=1$ and then extend $\deg_t(\cdot)$ to the rational function field $\mathbb{F}_q(t)$ by setting $\deg_t(P/Q):=\deg_t(P)-\deg_t(Q)$ with $P,Q\in\mathbb{F}_q[t]$ and $Q\neq 0$. We further extend $\deg_t(\cdot)$ to $\Mat_{m\times n}(\mathbb{F}_q(t))$ by setting $\deg_t(B):=\max_{i,j}\{\deg_t(B_{ij})\}$ for each $B=(B_{ij})\in\Mat_{m\times n}(\mathbb{F}_q(t))$.

    
    
\subsection{Siegel's lemma for function fields and lattice over $\mathbb{F}_q[t]$}
    In this section, we recall some tools from Diophantine geometry. We first briefly review a function field analogue of Bombieri-Vaaler's theorem \cite{BV83}, which was established by Thunder in \cite{Thu95}.  We start by introducing the notion of height.
    
    \begin{definition}\label{Def:Height}
        Let $n\in\mathbb{Z}_{>0}$. For each $\mathbf{x}=(x_1,\dots,x_n)\in L^n$ and $v\in M_L$, we set 
        $$\mathrm{ord}_v(\mathbf{x}):=\min_{1\leq i\leq n}\mathrm{ord}_v(x_i).$$
        Then, for each $0\neq\mathbf{x}\in L^n$, we obtain a divisor of $L$
        $$\mathrm{div}(\mathbf{x}):=\sum_{v\in M_L}\mathrm{ord}_v(\mathbf{x})\cdot v.$$
        The height of $\mathbf{x}$ is defined to be
        $$h(\mathbf{x}):=\frac{-1}{[L:k]}\deg(\mathrm{div}(\mathbf{x}))=\frac{-1}{[L:k]}\sum_{v\in M_L}\left(\min_{1\leq i\leq n}\mathrm{ord}_v(x_i)\cdot\deg(v)\right).$$
    \end{definition}
    
    Note that this height is projective in the sense that $h(c\cdot\mathbf{x})=h(\mathbf{x})$ for each $c\in L^\times$. We extend this height to matrices $B=(B_{ij})\in\Mat_{m\times n}(L)$ with $0<\rank(B)=m<n$ via the Grassmann coordinates. To be more precise, if $J\subset\{1,\dots,n\}$ is a subset with cardinality $|J|=m$, then we set $B_J:=(B_{ij})\in\Mat_{m\times m}(L)$ with $1\leq i\leq m$ and $j\in J$. The height of $B$ is defined by
    $$h(B):=\frac{-1}{[L:k]}\sum_{v\in M_L}\left(\min_J\mathrm{ord}_v(\det(B_J))\right)$$
    where $J$ runs over all $J\subset\{1,\dots,n\}$ with $|J|=m$.
    The following theorem is a special case of \cite[Cor.~2]{Thu95}.
    
    \begin{theorem}\label{Thm:Function_Fields_BV}
        Let $B\in\Mat_{m\times n}(k)$ with $0<\rank(B)=m<n$. Then, there are $k$-linearly independent vectors $\mathbf{x}_1\dots,\mathbf{x}_{n-m}\in \Mat_{n\times 1}(k)$ with $B\cdot\mathbf{x}_i=0$ satisfying that $$\sum_{i=1}^{n-m}h(\mathbf{x}_i)<h(B).$$
    \end{theorem}

    For $B=(B_{ij})\in\Mat_{m\times n}(\mathbb{F}_q[t])$, we recall that $\deg_t(B)=\max_{i,j}\{\deg_t(B_{ij})\}$. We state an immediate consequence of Theorem~\ref{Thm:Function_Fields_BV} as follows.
    
    \begin{corollary}\label{Cor:Function_Fields_BV}
        Let $B\in\Mat_{m\times n}(\mathbb{F}_q[t])$ with $0<\rank(B)=m<n$. Then, there are $\mathbb{F}_q[t]$-linearly independent $\mathbf{x}_1,\dots,\mathbf{x}_{n-m}\in\Mat_{n\times 1}(\mathbb{F}_q[t])$ with $B\cdot\mathbf{x}_i=0$ satisfying that
        $$\deg_t(\mathbf{x}_i)<m\cdot\deg_t(B).$$
    \end{corollary}

    Inspired by the proof of Masser in \cite{Mas88}, we introduce a basic property for free $\mathbb{F}_q[t]$-modules of finite rank. More precisely, the following lemma is inspired by \cite{Cas59} and it is crucial in the proof of Theorem~\ref{Thm:Main_Thm}.

    \begin{lemma}[cf.~{\cite[Lem.~8~(p.135)]{Cas59}}]\label{Lem:Lattice_2}
        Let $\nu\in\mathbb{Z}_{>0}$ and $\Lambda_0\subset\Lambda\subset\mathbb{F}_q[t]^\nu$ be two free $\mathbb{F}_q[t]$-modules of finite rank with $\rank_{\mathbb{F}_q[t]}\Lambda_0=\rank_{\mathbb{F}_q[t]}\Lambda=\nu$. Suppose that $\{\beta_1,\dots,\beta_\nu\}$ forms an $\mathbb{F}_q[t]$-basis of $\Lambda_0$. Then $\Lambda$ admits an $\mathbb{F}_q[t]$-basis $\{\gamma_1,\dots,\gamma_\nu\}$ such that for any $1\leq j\leq \nu$
        $$\deg_t(\gamma_j)\leq\max_{i\leq j}\{\deg_t(\beta_i)\}.$$
    \end{lemma}
    
    Note that most arguments are parallel to the proof of \cite[Thm.~I~(p.11)]{Cas59} and \cite[Lem.~8~(p.135)]{Cas59}. The major difference is that the ultrametric analysis here leads to a better estimation. We omit the proof here and we refer readers to \cite{Cas59} for the detailed arguments which essentially leads to Lemma~\ref{Lem:Lattice_2}.

\subsection{Anderson dual $t$-motives of Drinfeld $\mathbb{F}_q[t]$-modules}
    In this section, We recall the notion of Anderson dual $t$-motives. Here we will call dual $t$-motives by following \cite[Def.~4.4.1]{ABP04}. Note that the same object is called Anderson $t$-motives in \cite[Def.3.4.1]{Pap08}. Then we review the essential properties for dual $t$-motives associated to Drinfeld $\mathbb{F}_q[t]$-modules. For further information of these objects, one can consult \cite{BP20,HJ20,NP21}.

    For $n\in\mathbb{Z}$, we define the $n$-fold Frobenius twisting map
    \begin{align*}
	    \mathbb{C}_{\infty}((t))&\rightarrow\mathbb{C}_{\infty}((t))\\
	    f:=\sum_{i}a_it^i&\mapsto \sum_{i}a_i^{q^{n}}t^i=:f^{(n)}.
    \end{align*}
    We denote by $\overline{k}[t, \sigma]$ the non-commutative $\overline{k}[t]$-algebra generated by $\sigma$ subject to the following relation:
    \[
        \sigma f=f^{(-1)}\sigma, \quad f\in\overline{k}[t].
    \]
    Note that $\overline{k}[t,\sigma]$ contains $\overline{k}[t]$, $\overline{k}[\sigma]$, and its center is $\mathbb{F}_q[t]$. The definition of dual $t$-motives is given as follows.
    
    \begin{definition}
        A dual $t$-motive is a left  $\overline{k}[t, \sigma]$-module $\mathcal{M}$ satisfying that
        \begin{itemize}
            \item[(i)] $\mathcal{M}$ is a free left $\overline{k}[t]$-module of finite rank.
            \item[(ii)] $\mathcal{M}$ is a free left $\overline{k}[\sigma]$-module of finite rank.
            \item[(iii)] $(t-\theta)^n \mathcal{M}\subset \sigma \mathcal{M}$ for any sufficiently large integer $n$.
        \end{itemize}
    \end{definition}
    
    Now we recall the notion of \emph{$t$-frame}. Let $\mathcal{M}$ be a dual $t$-motive with $\overline{k}[t]$-basis $\{m_1,\dots,m_r\}$. Then there is an unique matrix $\Phi\in\Mat_r(\overline{k}[t])$ such that
    \[
        \sigma(m_1,\dots,m_r)^\tr=\Phi(m_1,\dots,m_r)^\tr.
    \]
    Now we define a map
    \begin{align*}
        \iota:\Mat_{1\times r}(\overline{k}[t])&\to\mathcal{M}\\
        (a_1,\dots,a_r)&\mapsto a_1m_1+\cdots+a_rm_r.
    \end{align*}
    We call the pair $(\iota,\Phi)$ a $t$-frame of $\mathcal{M}$. Note that for any $(a_1,\dots,a_r)\in\Mat_{1\times r}(\overline{k}[t])$, we have (cf.~\cite[Prop.~3.2.2]{NP21})
    \begin{align*}
        \sigma\iota(a_1,\dots,a_r)&=\sigma(a_1,\dots,a_r)(m_1,\dots,m_r)^\tr\\
        &=(a_1^{(-1)},\dots,a_r^{(-1)})\sigma(m_1,\dots,m_r)^\tr\\
        &=(a_1^{(-1)},\dots,a_r^{(-1)})\Phi(m_1,\dots,m_r)^\tr\\
        &=\iota((a_1^{(-1)},\dots,a_r^{(-1)})\Phi).
    \end{align*}
    
    Let $L/k$ be a finite extension and $E=(\mathbb{G}_{a/L},\phi)$ be a Drinfeld $\mathbb{F}_q[t]$-module of rank $r$ defined over $L$ with $\phi_t=\theta+\kappa_1\tau+\cdots+\kappa_r\tau^r\in L[\tau]$. In what follows, we explain how to associate a $\overline{k}[t,\sigma]$-module $\mathcal{M}_E$ for $E$.  Let $\mathcal{M}_E=\overline{k}[\sigma]$. It naturally has a left $\overline{k}[\sigma]$-module structure. The left $\overline{k}[t]$-module of $\mathcal{M}_E$ is given by the following setting: for each $m\in\mathcal{M}_E$, we define
    \[
        tm:=m\phi_t^\star:=m\left(\theta+\kappa_1^{(-1)}\sigma+\cdots+\kappa_r^{(-r)}\sigma^r\right).
    \]
    It is clear that $\mathcal{M}_E$ is also free of rank $1$ over $\overline{k}[\sigma]$ and it is a straightforward computation that
    \[
        (t-\theta)\mathcal{M}_E\subset\sigma\mathcal{M}_E.
    \]
    Note that $\mathcal{M}_E$ is free of rank $r$ over $\overline{k}[t]$, and $\{1,\sigma,\dots,\sigma^{r-1}\}$ forms a $\overline{k}[t]$-basis of $\mathcal{M}_E$. This $\ok[t]$-basis induces a $t$-frame $(\iota,\Phi_E)$ of $\mathcal{M}_E$, where $\Phi_E\in\Mat_{r}(\overline{k}[t])$ is given by
     \begin{equation}\label{Eq:Phi_E}
        \Phi_E=\begin{pmatrix}
            0 & 1 & & \\
            \vdots & & \ddots \\
            0 & & & 1\\
            \frac{t-\theta}{\kappa_r^{(-r)}} & \frac{-\kappa_1^{(-1)}}{\kappa_r^{(-r)}} & \cdots &
            \frac{-\kappa_{r-1}^{(-r+1)}}{\kappa_r^{(-r)}} 
        \end{pmatrix}\in\Mat_{r}(\overline{k}[t]).
    \end{equation}
    
    Now we are going to explain how to recover the $\mathbb{F}_q[t]$-module $E(\ok)$ from its associated dual $t$-motive $\mathcal{M}_E$. Let $m=\sum_{i=0}^{n}a_i\sigma^i\in\mathcal{M}_E$ with $a_i\in\ok$. Then we define
    \[
        \epsilon_0(m):=a_0\in\overline{k},~
        \epsilon_1(m):=\sum_{i=0}^na_i^{(i)}\in\overline{k}.
    \]
    Note that $\epsilon_0:\mathcal{M}_E\to\overline{k}$ is a $\overline{k}$-linear map and $\epsilon_1:\mathcal{M}_E\to\overline{k}$ is an $\mathbb{F}_q$-linear map. We have the following lemma due to Anderson.
    
    \begin{lemma}[Anderson,~{\cite[Prop.~2.5.8]{HJ20}},~{\cite[Lem.~3.1.2]{NP21}}]\label{Lem:DtMotives_to_tModules}
        For any $a\in\mathbb{F}_q[t]$, we have the following commutative diagrams with exact rows:
        \[
        \begin{tikzcd}
            0 \arrow[r] & \mathcal{M}_E \arrow[d, "a(\cdot)"] \arrow[r, "\sigma(\cdot)"] & \mathcal{M}_E \arrow[d, "a(\cdot)"] \arrow[r, "\epsilon_0"] & \overline{k} \arrow[d, "\partial\phi_a(\cdot)"] \arrow[r] & 0 \\
            0 \arrow[r] & \mathcal{M}_E \arrow[r, "\sigma(\cdot)"] & \mathcal{M}_E \arrow[r, "\epsilon_0"] & \overline{k} \arrow[r] & 0
        \end{tikzcd}
        \]
        and
        \[
        \begin{tikzcd}
            0 \arrow[r] & \mathcal{M}_E \arrow[d, "a(\cdot)"] \arrow[r, "(\sigma-1)(\cdot)"] & \mathcal{M}_E \arrow[d, "a(\cdot)"] \arrow[r, "\epsilon_1"] & \overline{k} \arrow[d, "\phi_a(\cdot)"] \arrow[r] & 0 \\
            0 \arrow[r] & \mathcal{M}_E \arrow[r, "(\sigma-1)(\cdot)"] & \mathcal{M}_E \arrow[r, "\epsilon_1"] & \overline{k} \arrow[r] & 0.
        \end{tikzcd}
        \]
        In particular, $\epsilon_0$ and $\epsilon_1$ induce isomorphisms:
        \[
            \epsilon_0:\mathcal{M}_E/\sigma\mathcal{M}_E\cong \Lie(E)(\overline{k}),~
            \epsilon_1:\mathcal{M}_E/(\sigma-1)\mathcal{M}_E\cong E(\overline{k})
        \]
        where $\epsilon_0$ is $\overline{k}[t]$-linear and $\epsilon_1$ is $\mathbb{F}_q[t]$-linear.
    \end{lemma}
   
    Let $P\in E(L)$ be an element in $L$. Then 
    \[
        \epsilon_1^{-1}(P)=\iota(P,0,\dots,0)+(\sigma-1)\mathcal{M}_E\in\mathcal{M}_E/(\sigma-1)\mathcal{M}_E
    \]
    where $\epsilon_1$ is the induced isomorphism of $\mathbb{F}_q[t]$-modules given in Lemma~\ref{Lem:DtMotives_to_tModules} between $\mathcal{M}_E/(\sigma-1)\mathcal{M}_E$ and $E(\overline{k})$. The crucial point of our strategy in the present paper is to identify algebraic points on Drinfeld modules to elements in the quotient of its associated dual $t$-motives via $\epsilon_1$. Then we study the corresponding question in the setting of dual $t$-motives.

\section{The main theorem}
%
%
    
\subsection{The Frobenius difference equation}
    The primary goal of this subsection is to establish Lemma~\ref{Lem:Reduction_1}, which shows the equivalence between the $\mathbb{F}_q[t]$-linear relations among algebraic points on Drinfeld $\mathbb{F}_q[t]$-modules and the solvability of a Frobenius difference equation. We mention that Ho in his master thesis \cite[Thm.~4]{Ho20} established a similar result for $A$-valued points on Drinfeld $\mathbb{F}_q[t]$-modules $E$ defined over $A$ by following the ideas of \cite{Cha16}. In the present paper, we adopt a simplified approach avoiding the usage of $\Ext^1$-modules (cf.~\cite{CPY19}). Moreover, our result includes algebraic points in $E(L)$ where $E$ is a Drinfeld $\mathbb{F}_q[t]$-modules defined over $L$ and $L/k$ is a finite extension.
    \begin{lemma}\label{Lem:Reduction_1}
        Let $L/k$ be a finite extension. Let $E=(\mathbb{G}_{a/L},\phi)$ be a Drinfeld $\mathbb{F}_q[t]$-module defined over $L$ with $\phi_t=\theta+\kappa_1\tau+\cdots+\kappa_r\tau^r\in L[\tau]$ and $\kappa_r\neq 0$. Let $P_1,\dots,P_\ell\in E(L)$ be distinct non-zero elements in $L$ and $a_1,\dots,a_\ell\in\mathbb{F}_q[t]$ not all zero. Then the following two assertions are equivalent.
        \begin{enumerate}
            \item $\phi_{a_1}(P_1)+\cdots+\phi_{a_\ell}(P_\ell)=0$.
            \item There exists $g\in L[t]$ such that
            $$\kappa_rg^{(r)}+\cdots+\kappa_1g^{(1)}-(t-\theta)g=a_1P_1+\cdots+a_\ell P_\ell.$$
        \end{enumerate}
        Furthermore, if $P_1,\dots,P_\ell$ are linearly independent over $\mathbb{F}_q$, then the polynomial $g\in L[t]$ in the statement (2) must be non-zero.
    \end{lemma}
    
    \begin{proof}
        Recall that the dual $t$-motive $\mathcal{M}_E=\ok[\sigma]$ associated with the Drinfeld $\mathbb{F}_q[t]$-module $E$ is a free $\overline{k}[t]$-module of rank $r$. Let $\{1,\sigma,\dots,\sigma^{r-1}\}$ be a $\overline{k}[t]$-basis of $\mathcal{M}_E$. Then, we have the associated $t$-frame $(\iota,\Phi_E)$, where $\Phi_E\in\Mat_{r}(\overline{k}[t])$ is given in \eqref{Eq:Phi_E}. Furthermore, we have the following five equivalences (cf. \cite{Cha16,KL16,Ho20}):
        \begin{enumerate}
            \item $\phi_{a_1}(P_1)+\cdots+\phi_{a_\ell}(P_\ell)=0$ in $E(L)$.
            \item $\iota(a_1P_1+\cdots+a_\ell P_\ell,0,\dots,0)=0$ in $\mathcal{M}_E/(\sigma-1)\mathcal{M}_E$.
            \item There exist $\delta_1,\dots,\delta_r\in\overline{k}[t]$ such that
            \begin{equation}\label{Eq:Equi_3}
                (a_1P_1+\cdots+a_\ell P_\ell,0,\dots,0)=(-\delta_1^{(-1)},\dots,-\delta_r^{(-1)})\Phi_E-(-\delta_1,\dots,-\delta_r).
            \end{equation}
            \item There exist $\delta_1,\dots,\delta_r\in\overline{k}[t]$ such that
            \begin{equation}\label{Eq:Equi_4}
                \begin{cases}
                \delta_1&=\delta_r^{(-1)}\left(\frac{t-\theta}{\kappa_r^{(-r)}}\right)+a_1P_1+\cdots+a_\ell P_\ell\\
                \delta_2^{(1)}&=\delta_1+\delta_r\left(\frac{-\kappa_1}{\kappa_r^{(-r+1)}}\right)\\
                &\vdots\\ 
                \delta_r^{(r-1)}&=\delta_{r-1}^{(r-2)}+\delta_r^{(r-2)}\left(\frac{-\kappa_{r-1}}{\kappa_r^{(-1)}}\right)
                \end{cases}.
            \end{equation}
            \item There exists $g\in L[t]$ such that
            \begin{equation}\label{Eq:Equi_5}
                \kappa_rg^{(r)}+\cdots+\kappa_1g^{(1)}-(t-\theta)g=a_1P_1+\cdots+a_\ell P_\ell.
            \end{equation}
        \end{enumerate}
        
        The equivalence of the first and the second statement comes from Lemma~\ref{Lem:DtMotives_to_tModules} and our specific choice of the $\overline{k}[t]$-basis of $\mathcal{M}_E$. To see the equivalence of the second and the third statement, we note that the second statement is equivalent to the existence of $m\in\mathcal{M}_E$ such that $\iota(a_1P_1+\cdots+a_\ell P_\ell,0,\dots,0)=(\sigma-1)m$, and this is equivalent to the third statement since $\mathcal{M}_E$ is a free $\overline{k}[t]$-module. The equivalence of the third and the fourth statement comes from the direct computations and comparing each entry on both sides of \eqref{Eq:Equi_3}.        
        Now, it remains to explain the last two equivalences.
        
        \textbf{Claim 1: } $(4)\implies(5)$.
        \begin{proof}[{Proof of Claim 1}]\renewcommand{\qedsymbol}{}
            Suppose that there exist $\delta_1,\dots,\delta_r\in\overline{k}[t]$ such that (\ref{Eq:Equi_4}) holds. We set $g:=\frac{\delta_r^{(-1)}}{\kappa_r^{(-r)}}\in\overline{k}[t]$. Let $F:=a_1P_1+\cdots+a_\ell P_\ell$. Then, the first equality of (\ref{Eq:Equi_4}) is equivalent to $\delta_1=(t-\theta) g+F$. Since $g=\frac{\delta_r^{(-1)}}{\kappa_r^{(-r)}}$ and $\delta_1=(t-\theta) g+F$, the second equality of (\ref{Eq:Equi_4}) is equivalent to $\delta_2^{(1)}=(t-\theta) g+F-\kappa_1g^{(1)}$. We inductively continue this replacement, and finally, we conclude that the last equality of (\ref{Eq:Equi_4}) is equivalent to $\delta_r^{(r-1)}=(t-\theta) g+F-\kappa_1g^{(1)}-\cdots-\kappa_{r-1}g^{(r-1)}$, which is equivalent to $(t-\theta)g+F=\kappa_rg^{(r)}+\cdots+\kappa_1g^{(1)}$. Thus, it shows that $g=\frac{\delta_r^{(-1)}}{\kappa_r^{(-r)}}$ satisfies \eqref{Eq:Equi_5} as desired.
            To see that the coefficients of $g=\frac{\delta_r^{(-1)}}{\kappa_r^{(-r)}}$ are actually in $L$, we first compare $\deg_t(\cdot)$ on both sides of (\ref{Eq:Equi_5}). Since the right-hand side of (\ref{Eq:Equi_5}) has a degree less than or equal to $\deg_t(g)$, we may assume that $\deg_t(g)=n$ and $\deg_t(F)=n+1$. Now we express
            $$g=g_0+g_1t+\cdots+g_nt^n,~g_i\in\overline{k}$$
            and
            $$F=F_0+F_1t+\cdots+F_{n+1}t^{n+1},~F_i\in L.$$
            Comparing the coefficient of $t^i$ on both sides of (\ref{Eq:Equi_5}), we obtain
            $$g_n+F_{n+1}=0$$
            and
            $$g_{i-1}-\theta g_i+F_i=\kappa_r g_i^{q^r}+\cdots+\kappa_1 g_i^{q}$$
            for each $1\leq i\leq n$. Since $F_i\in L$ for each $0\leq i\leq n+1$, we inductively conclude that $g_i\in L$ for each $0\leq i\leq n$. Hence, $g\in L[t]$ as desired.
        \end{proof}
        
        \textbf{Claim 2: } $(5)\implies(4)$.
        \begin{proof}[{Proof of Claim 2}]\renewcommand{\qedsymbol}{}
            Suppose that there exists $g\in L[t]$ such that (\ref{Eq:Equi_5}) holds. We set $\delta_r:=\kappa_r^{(-r+1)} g^{(1)}\in\overline{k}[t]$. It is not hard to see that once $\delta_r$ is fixed, $\delta_1,\dots,\delta_{r-1}$ are automatically defined by the first $r-1$ equations of (\ref{Eq:Equi_4}). Thus, it remains to check the compatibility of the last equality of (\ref{Eq:Equi_4}). But it can be easily deduced from the first $r-1$ equations of (\ref{Eq:Equi_4}) and (\ref{Eq:Equi_5}). The desired result now follows immediately.
        \end{proof}
    \end{proof}
    
    \begin{remark}
        Let
        \[
            \mathbb{T}:=\{f=\sum_{i=0}^\infty f_it^i\in\mathbb{C}_\infty\llbracket t\rrbracket\mid \lim_{i\to\infty}|f_i|_\infty=0\}.
        \]
        If we set $a_1P_1+\cdots+a_\ell P_\ell=0$ in the second equivalent statement of Lemma~\ref{Lem:Reduction_1}, then by \cite[Sec.~4.2]{Pel08} (see also \cite[Prop.~6.2]{EGP14}) the homogeneous Frobenius difference equation 
        \[
            \kappa_rg^{(r)}+\cdots+\kappa_1g^{(1)}-(t-\theta)g=0
        \]
        admits $r$ linearly independent solutions in $\mathbb{T}$ over $\mathbb{F}_q(t)$ which are given by the Anderson generating functions of the Drinfeld module $E$. Lemma~\ref{Lem:Reduction_1} says that for $\mathbb{F}_q$-linearly independent $P_1,\dots,P_\ell\in E(L)$, the existence of non-trivial $\mathbb{F}_q[t]$-linear relations $\phi_{a_1}(P_1)+\cdots+\phi_{a_\ell}(P_\ell)=0$ is equivalent to the existence of the non-zero solutions $g\in L[t]$ of the non-homogeneous Frobenius difference equation
        \[
            \kappa_rg^{(r)}+\cdots+\kappa_1g^{(1)}-(t-\theta)g=a_1P_1+\cdots+a_\ell P_\ell.
        \]
    \end{remark}
    
\subsection{Linear relations among algebraic points on Drinfeld modules}
    In what follows, we aim to study linear relations among algebraic points on Drinfeld modules based on Lemma~\ref{Lem:Reduction_1}. Let $L/k$ be a finite extension. If $H(t):=H_0+H_1t+\cdots+H_mt^m\in L[t]$ with $H_i\in L$, then we set $\mathrm{ord}_v(H):=\min_{1\leq i\leq m}\{\mathrm{ord}_v(H_i)\}$ for each $v\in M_L$. Let $E=(\mathbb{G}_{a/L},\phi)$ be a Drinfeld $\mathbb{F}_q[t]$-module of rank $r$ defined over $L$ with $\phi_t=\theta+\kappa_1\tau+\cdots+\kappa_r\tau^r\in L[\tau]$. For each $v\in M_L$, we define (cf. \cite[Sec.~2]{Tag93})
    \[
        \mathcal{S}_{E,v}:=\max_{1\leq j\leq r}\{\frac{\mathrm{ord}_v(\kappa_r)-\mathrm{ord}_v(\kappa_j)}{q^r-q^j},\frac{\mathrm{ord}_v(\kappa_r)-\mathrm{ord}_v(t-\theta)}{q^r-1}\}.
    \]
    Furthermore, we define a divisor over $L$ associated to the Drinfeld $\mathbb{F}_q[t]$-module $E$
    \begin{equation}\label{Eq:slope_divisor}
        D_E:=\sum_{v\in M_L}\lceil\mathcal{S}_{E,v}\rceil\cdot v,
    \end{equation}
    where $\lceil X\rceil$ denote the smallest integer which exceeds $X$. 
    
    Recall that the associate polynomial $\mathcal{P}_E$ of $E$ defined in \eqref{Eq:Newton_Polygon_of_Drinfeld} is given by
    \[
        \mathcal{P}_E(x)=(t-\theta)x+\kappa_1x^q+\cdots+\kappa_rx^{q^r}\in L[t][x].
    \]
    Then the quantity $\mathcal{S}_{E,v}$ has the following interpretation: for each place $v\in M_L$, we define $\mathcal{N}_{E,v}$ to be the Newton polygon of $P_E(x)$ at the place $v$. More precisely, it is the Newton polygon consisting of the points $(1,\mathrm{ord}_v(t-\theta)),(q,\mathrm{ord}_v(\kappa_1)),\dots,(q^r,\mathrm{ord}_v(\kappa_r))$. Then one can derive by the direct calculation that $\mathcal{S}_{E,v}$ is the slope of the bottom rightmost edge on $\mathcal{N}_{E,v}$ passing through the point $(q^r,\mathrm{ord}_v(\kappa_r))$. 

    The next lemma asserts that the degree of $D_E$ is invariant under $L$-isomorphisms. In other words, $\deg(D_E)$ is well-defined on the $L$-isomorphism classes of $\mathbb{F}_q[t]$-modules.
    \begin{lemma}\label{Lem:Invariant}
        Let $u$ be an $L$-isomorphism of $E$. Then
        \[
            D_{u(E)}=D_E-\mathrm{div}(u).
        \]
        In particular, we have $\deg(D_{u(E)})=\deg(D_{E})$.
    \end{lemma}
    
    \begin{proof}
        We denote by $u(E)=(\mathbb{G}_{a/L},\phi')$ to be the Drinfeld $\mathbb{F}_q[t]$-module of rank $r$ defined over $L$ with $\phi'_t=\theta+\kappa'_1\tau+\cdots+\kappa'_r\tau^r\in L[\tau]$. Since $u\phi_t=\phi_t'u$, we must have $\kappa_j=u^{q^j-1}\kappa_j'$ for all $1\leq j\leq r$. Thus, for $1\leq j\leq r$, we have
        $$\frac{\mathrm{ord}_v(\kappa'_r)-\mathrm{ord}_v(\kappa'_j)}{q^r-q^j}=\frac{\mathrm{ord}_v(\kappa_r)-\mathrm{ord}_v(\kappa_j)}{q^r-q^j}-\mathrm{ord}_v(u)$$
        and
        $$\frac{\mathrm{ord}_v(\kappa'_r)-\mathrm{ord}_v(t-\theta)}{q^r-1}=\frac{\mathrm{ord}_v(\kappa_r)-\mathrm{ord}_v(t-\theta)}{q^r-1}-\mathrm{ord}_v(u).$$
        As $\mathrm{ord}_v(u)\in\mathbb{Z}$, we derive that for all $v\in M_L$
        $$D_{u(E)}=D_E-\mathrm{div}(u).$$
        The desired result now follows.
    \end{proof}
    
    We are ready to prove our first main result in this subsection.
    
    \begin{theorem}\label{Thm:Reduction_2}
        Let $L/k$ be a finite extension. Let $E=(\mathbb{G}_{a/L},\phi)$ be a Drinfeld $\mathbb{F}_q[t]$-module defined over $L$ with $\phi_t=\theta+\kappa_1\tau+\cdots+\kappa_r\tau^r\in L[\tau]$ and $\kappa_r\neq 0$. Let $P_1,\dots,P_\ell\in E(L)$ be distinct non-zero elements in $L$. Then there exists a matrix $B\in\Mat_{m\times n}(\mathbb{F}_q[t])$ with $\deg_t(B)\leq 1$ and $0< m:=\rank(B)<n:=d+\ell$  such that the canonical projection
        \begin{align*}
            \pi:\Gamma:=\{\mathbf{x}\in\mathbb{F}_q[t]^{(d+\ell)}\mid B\mathbf{x}^\tr=0\}&\twoheadrightarrow G=\{(a_1,\dots,a_\ell)\in\mathbb{F}_q[t]^\ell\mid \sum_{i=1}^\ell\phi_{a_i}(P_i)=0\}\\
            (g_1,\dots,g_d,a_1,\dots,a_\ell)&\mapsto (a_1,\dots,a_\ell)
        \end{align*}
        is a well-defined surjective $\mathbb{F}_q[t]$-module homomorphism.
    \end{theorem}
    
    \begin{proof}
        If $P_1,\dots,P_\ell$ are $\mathbb{F}_q[t]$-linearly independent, then the statement is clear. Thus, we may assume that $P_1,\dots,P_\ell$ are $\mathbb{F}_q[t]$-linearly dependent.
        For each $v\in M_L$ we set
        $$C_v:=\underset{1\leq j\leq r-1}{\underset{1\leq i\leq \ell}{\min}}
        \{\mathrm{ord}_v(P_i),\lfloor\frac{\mathrm{ord}_v(\kappa_j)-\mathrm{ord}_v(\kappa_r)}{q^r-q^j}\rfloor,\lfloor\frac{\mathrm{ord}_v(t-\theta)-\mathrm{ord}_v(\kappa_r)}{q^r-1}\rfloor\}.$$
        It is not hard to see that $C_v=0$ for all but finitely many $v\in M_L$.
        Let $a_1,\dots,a_\ell\in\mathbb{F}_q[t]$ not all zero. Then according to Lemma~\ref{Lem:Reduction_1}, $\sum_{i=1}^\ell\phi_{a_i}(P_i)=0$ is equivalent to 
        the existence of $g\in L[t]$ such that
        \begin{equation}\label{Eq:Reduc_1}
            \kappa_r g^{(r)}+\cdots+\kappa_1 g^{(1)}-(t-\theta)g=a_1P_1+\cdots+a_\ell P_\ell.
        \end{equation}
        Let $F:=a_1P_1+\cdots+a_\ell P_\ell$. Then we claim that $\mathrm{ord}_v(g)\geq C_v$ and $\mathrm{ord}_v(F)\geq C_v$. 
        
        To prove this claim, we first note that $\mathrm{ord}_v(F)\geq C_v$ is clear since
        $$\mathrm{ord}_v(F)=\mathrm{ord}_v(\sum a_iP_i)\geq \min\{\mathrm{ord}_v(P_i)\}\geq C_v.$$
        To prove $\mathrm{ord}_v(g)\geq C_v$, we first note that $\mathrm{ord}_v(t-\theta)=\min\{0,\mathrm{ord}_v(\theta)\}\leq 0$. Now suppose on the contrary that $\mathrm{ord}_v(g)<C_v$. Then we have
        $$\mathrm{ord}_v(g)<C_v\leq \mathrm{ord}_v(F)-\mathrm{ord}_v(t-\theta).$$
        Here we use the first part of the claim and the fact that $\mathrm{ord}_v(t-\theta)\leq 0$. In particular, we obtain
        $$\mathrm{ord}_v((t-\theta) g)=\mathrm{ord}_v(t-\theta)+\mathrm{ord}_v(g)<\mathrm{ord}_v(F).$$
        
        On the other hand, we also have
        $$\mathrm{ord}_v(g)<\frac{\mathrm{ord}_v(\kappa_j)-\mathrm{ord}_v(\kappa_r)}{q^r-q^j},~\mbox{for each}~1\leq j\leq r-1$$
        and hence we have
        $$\mathrm{ord}_v(\kappa_r g^{(r)})=\mathrm{ord}_v(\kappa_r)+q^r\mathrm{ord}_v(g)<\mathrm{ord}_v(\kappa_j)+q^j\mathrm{ord}_v(g)=\mathrm{ord}_v(\kappa_j g^{(j)}).$$
        It follows by comparing $\mathrm{ord}_v(\cdot)$ of both sides of (\ref{Eq:Reduc_1}) that
        $$\mathrm{ord}_v(t-\theta)+\mathrm{ord}_v(g)=\mathrm{ord}_v(\kappa_r)+q^r\mathrm{ord}_v(g).$$
        Thus, we obtain 
        $$\mathrm{ord}_v(g)=\frac{\mathrm{ord}_v(t-\theta)-\mathrm{ord}_v(\kappa_r)}{q^r-1}$$
        which contradicts to our assumption that $\mathrm{ord}_v(g)<C_v$. 
        
        Now we define 
        $$D:=\sum_{v\in M_L}\left(-C_v\right)\cdot v.$$
        Then $D$ defines a divisor of $L$ and we claim that
        \begin{equation}\label{Eq:divisor_identity}
            D=\left(-\mathrm{div}(P_1,\dots,P_\ell)\right)\vee D_E
        \end{equation}
        where $\mathrm{div}(P_1,\dots,P_\ell)$ and $D_E$ are divisors defined in Definition~\ref{Def:Height} and (\ref{Eq:slope_divisor}) respectively. The operator $\vee$ is defined in (\ref{Eq:max_divisors}). To verify (\ref{Eq:divisor_identity}), we recall that
        \[
            \mathrm{div}(P_1,\dots,P_\ell)=\sum_{v\in M_L}\left(\min_{1\leq i\leq\ell}\mathrm{ord}_v(P_i)\right)\cdot v
        \]
        and
        \[
            D_E=\sum_{v\in M_L}\max_{1\leq j\leq r}\{\lceil\frac{\mathrm{ord}_v(\kappa_r)-\mathrm{ord}_v(\kappa_j)}{q^r-q^j}\rceil,\lceil\frac{\mathrm{ord}_v(\kappa_r)-\mathrm{ord}_v(t-\theta)}{q^r-1}\rceil\}\cdot v.
        \]
        Thus, one derives that
        \begin{align*}
            \left(-\mathrm{div}(P_1,\dots,P_\ell)\right)\vee D_E&=\sum_{v\in M_L}\max\{-\left(\min_{1\leq i\leq\ell}\mathrm{ord}_v(P_i)\right),\lceil\mathcal{S}_{E,v}\rceil\}\cdot v\\
            &=\sum_{v\in M_L}-\min\{\left(\min_{1\leq i\leq\ell}\mathrm{ord}_v(P_i)\right),\lfloor-\mathcal{S}_{E,v}\rfloor\}\cdot v\\
            &=\sum_{v\in M_L}\left(-C_v\right)\cdot v\\
            &=D
        \end{align*}
        which proves the desireed claim.
        
        We have shown that $g,F\in\mathcal{L}(D)\otimes_{\mathbb{F}_q}\mathbb{F}_q[t]$ in the sense that all coefficients of $g$ and $F$ lie in $\mathcal{L}(D)$ when we regard $g$ and $F$ as polynomials in $L[t]$. 
        Let $\{\beta_1,\dots,\beta_d\}$ be an $\mathbb{F}_q$-basis of $\mathcal{L}(D)$ where $d:=\dim_{\mathbb{F}_q}\mathcal{L}(D)$. We may express 
        $$g=g_1\beta_1+\cdots+g_d\beta_d,~\mbox{ for some }g_1\dots,g_d\in\mathbb{F}_q[t]$$
        and
        $$F=\left(\sum_{i=1}^\ell a_iP_{i1}\right)\beta_1+\cdots+\left(\sum_{i=1}^\ell a_iP_{id}\right)\beta_d,~\mbox{ for some }P_{ij}\in\mathbb{F}_q.$$
        It follows from (\ref{Eq:Reduc_1}) that
        \begin{equation*}
            \kappa_r\Big(g_1\beta_1^{q^r}+\cdots+g_d\beta_d^{q^r}\Big)+\cdots+\kappa_1\Big(g_1\beta_1^{q}+\cdots+g_d\beta_d^{q}\Big)-(t-\theta)\Big(g_1\beta_1+\cdots+g_d\beta_d\Big)\\
        \end{equation*}
        \begin{equation}\label{Eq:Linearized_1}
            =\left(\sum_{i=1}^\ell a_iP_{i1}\right)\beta_1+\cdots+\left(\sum_{i=1}^\ell a_iP_{id}\right)\beta_d .
        \end{equation}
        
        Now we set
        $$V:=\mathrm{Span}_{\mathbb{F}_q}\{\beta_i,~\theta\beta_i,~\kappa_j\beta_i^{q^j}\}_{1\leq i\leq d,~1\leq j\leq r}.$$ Let $\{\gamma_1,\dots,\gamma_{\Tilde{d}}\}\subset\{\beta_i,~\theta\beta_i,~\kappa_j\beta_i^{q^j}\}_{1\leq i\leq d,~1\leq j\leq r}$ be an $\mathbb{F}_q$-basis of $V$ where $\Tilde{d}:=\dim_{\mathbb{F}_q}V$.
        Then (\ref{Eq:Linearized_1}) becomes
        \begin{equation}\label{Eq:Linearized_2}
            H_1(g_1,\dots,g_d,a_1,\dots,a_\ell)\cdot\gamma_1+\cdots+H_{\Tilde{d}}(g_1,\dots,g_d,a_1,\dots,a_\ell)\cdot\gamma_{\Tilde{d}}=0
        \end{equation}
        where $H_i(X_1,\dots,X_{d+\ell})\in\mathbb{F}_q[t][X_1,\dots,X_{d+\ell}]$ is a homogeneous linear form in variables $X_1,\dots,X_{d+\ell}$ and $\deg_t(H_i)\leq 1$ for all $1\leq i\leq \Tilde{d}$. Since $\gamma_1,\dots,\gamma_{\Tilde{d}}$ are $\mathbb{F}_q$-linearly independent and one checks directly that $\mathbb{F}_q[t]$ and $L$ are linearly disjoint over $\mathbb{F}_q$, we obtain that $\gamma_1,\dots,\gamma_{\Tilde{d}}$ are $\mathbb{F}_q[t]$-linearly independent. Thus, $H_i(g_1,\dots,g_d,a_1,\dots,a_\ell)=0$ provides an $\mathbb{F}_q[t]$-linear system $B(g_1,\dots,g_d,a_1,\dots,a_\ell)^\tr=0$ for some $B\in\Mat_{m\times n}(\mathbb{F}_q[t])$ with $\deg_t(B)\leq 1$ and $0<m=\rank(B)<n=d+\ell$. Note that every solution $\mathbf{x}$ of $B\mathbf{x}^\tr=0$ gives a solution of (\ref{Eq:Reduc_1}) and vice versa. Thus, the desired result now follows immediately from Lemma~\ref{Lem:Reduction_1}.
    \end{proof}
    
    \begin{example}\label{Ex:Carlitz}
        Let $q>2$ and $C=(\mathbb{G}_{a/k},[\cdot])$ be the Carlitz module which is a Drinfeld module of rank $1$ defined over $k$ with the $\mathbb{F}_q$-linear algebra homomorphism
        \begin{align*}
            [\cdot]:\mathbb{F}_q[t]&\to k[\tau]\\
            t&\mapsto [t]:=\theta+\tau.
        \end{align*}
        Let $P_1=1,~P_2=\theta+1\in C(k)$. In what follows we would like to explain the $\mathbb{F}_q[t]$-linear relation
        \[
            [-t]P_1+[1]P_2=0
        \]
        by following the proof of Theorem~\ref{Thm:Reduction_2}. Let $a_1,a_2\in\mathbb{F}_q[t]$ not all zero. Then by Lemma~\ref{Lem:Reduction_1} we know that $[a_1]P_1+[a_2]P_2=0$ is equivalent to the existence of $g\in k[t]$ such that
        \[
            g^{(1)}-(t-\theta)g=a_1P_2+a_2P_2.
        \]
        By the same argument in the proof of Theorem~\ref{Thm:Reduction_2}, we know that the coefficients of $g$ as a polynomial in $k[t]$ must lie in the $\mathbb{F}_q$-vector space $\mathcal{L}(1\cdot(\infty))=\mathbb{F}_q+\mathbb{F}_q\cdot\theta$. Thus, the existence of $g\in k[t]$ is now equivalent to the existence of $g_0,~g_1\in\mathbb{F}_q[t]$ such that
        \[
            (g_0+g_1\theta^q)-(t-\theta)(g_0+g_1\theta)=(a_1+a_2)+a_2\theta.
        \]
        Since $q>2$, comparing with the $\mathbb{F}_q[t]$-coefficient of $1,~\theta$ and $\theta^q$ shows that
        \[
            \begin{cases}
                g_1=0\\
                (1-t)g_0=a_1+a_2\\
                a_2=g_0
            \end{cases}.
        \]
        Thus, the matrix $B$ in the statement of Theorem~\ref{Thm:Reduction_2} is described by
        \[
            B=\begin{pmatrix}
                0 & 1 & 0 & 0 \\
                1-t & 0 & 1 & 1 \\
                1 & 0 & 0 & 1
            \end{pmatrix}
        \]
        and
        \[
            \Gamma=\{\mathbf{x}=(g_0,g_1,a_1,a_2)\in\mathbb{F}_q[t]^4\mid B\mathbf{x}^\tr=0\}=\mathrm{Span}_{\mathbb{F}_q[t]}\{(1,0,-t,1)\}.
        \]
        In particular, the $\mathbb{F}_q[t]$-linear relations among $P_1$ and $P_2$ are generated by
        \[
            G=\{(a_1,a_2)\in\mathbb{F}_q[t]^2\mid[a_1]P_1+[a_2]P_2=0\}=\mathrm{Span}_{\mathbb{F}_q[t]}\{\pi(1,0,-t,1)=(-t,1)\}.
        \]
        Note that we can derive the same conclusion without the condition $q>2$. This condition is only used for simplifying the calculation.
    \end{example}
    
    \begin{proof}[Proof of Theorem~\ref{Thm:Main_Thm}]
        Let notation be the same as in Theorem~\ref{Thm:Reduction_2}.
        By Corollary~\ref{Cor:Function_Fields_BV}, there exist $\mathbb{F}_q[t]$-linearly independent vectors $\mathbf{x}_1,\dots,\mathbf{x}_{n-m}$ with entries in $\mathbb{F}_q[t]$ such that $\deg_t(\mathbf{x}_i)\leq\rank(B)\cdot\deg_t(B)< d+\ell$ and $B\cdot\mathbf{x}_i^\tr=0$ for each $1\leq i\leq n-m$. Let $\nu:=\rank_{\mathbb{F}_q[t]}G$. Since $G$ is a free $\mathbb{F}_q[t]$-module of rank $\nu$ and $\pi$ is surjective, there exists an $\mathbb{F}_q[t]$-linearly independent set 
        $$\{\mathbf{m}_1',\dots,\mathbf{m}_\nu'\}\subset\{\pi(\mathbf{x}_1),\dots,\pi(\mathbf{x}_{n-m})\}$$
        such that $G_0:=\mathrm{Span}_{\mathbb{F}_q[t]}\{\mathbf{m}_1',\dots,\mathbf{m}_\nu'\}\subset G$ is of finite index. In other words, we have $\rank_{\mathbb{F}_q[t]}G_0=\rank_{\mathbb{F}_q[t]}G=\nu$. Now we apply Lemma~\ref{Lem:Lattice_2} and we get an $\mathbb{F}_q[t]$-basis $\{\mathbf{m}_1,\dots,\mathbf{m}_\nu\}$ of $G$ such that $\deg_t(\mathbf{m}_i)\leq\max_{i=1}^\nu\{\deg_t(\mathbf{m}_i')\} <d+\ell$. 
    
        It remains to prove that $\deg(D)$ is invariant under $L$-isomorphisms of Drinfeld $\mathbb{F}_q[t]$-modules.
        Let $u$ be an $L$-isomorphism of $E$. We denote by $u(E)=(\mathbb{G}_{a/L},\phi')$ to be the Drinfeld $\mathbb{F}_q[t]$-module of rank $r$ defined over $L$ with $\phi'_t=\theta+\kappa'_1\tau+\cdots+\kappa'_r\tau^r\in L[\tau]$. If we set
        $$u(D)=\left(-\mathrm{div}(u(P_1),\dots,u(P_\ell))\right)\vee D_{u(E)},$$
        then we aim to show that $\deg(D)=\deg(u(D))$.
        
        To prove this equality, we first note that $\mathrm{ord}_v(u(P_i))=\mathrm{ord}_v(u)+\mathrm{ord}_v(P_i)$ for all $1\leq i\leq \ell$. Thus, we have 
        $$\mathrm{div}(u(P_1),\dots,u(P_\ell))=\mathrm{div}(P_1,\dots,P_\ell)+\mathrm{div}(u).$$
        Also, according to Lemma~\ref{Lem:Invariant}, we have
        \[
            D_{u(E)}=D_E-\mathrm{div}(u).
        \]
        Consequently, we get
        \begin{align*}
            u(D)&=\left(-\mathrm{div}(u(P_1),\dots,u(P_\ell))\right)\vee D_{u(E)}\\
            &=(\left(-\mathrm{div}(P_1,\dots,P_\ell)-\mathrm{div}(u)\right)\vee(D_E-\mathrm{div}(u))\\
            &=(\left(-\mathrm{div}(P_1,\dots,P_\ell)\right)\vee D_E)-\mathrm{div}(u)\\
            &=D-\mathrm{div}(u).
        \end{align*}
        Therefore, $\deg(D)=\deg(u(D))$ as the degree of the principal divisor is zero. Thus we complete the proof.
    \end{proof}
    
    \begin{remark}\label{Rmk:Optimality}
        Let $C=(\mathbb{G}_{a/k},[\cdot])$ be the Carlitz module given in Example~\ref{Ex:Carlitz}. Let $P_1=1$ and $P_n=[t^{n-1}]P_1$ for each $n\geq 2$. Then it is not hard to see that for each $n\geq 2$
        \[
            G_n=\{(a_1,a_n)\in\mathbb{F}_q[t]^2\mid [a_1]P_1+[a_n]P_n=0\}=\mathrm{Span}_{\mathbb{F}_q[t]}\{(-t^{n-1},1)\}
        \]
        and $\deg_t(-t^{n-1},1)=n-1$. 
        
        Now we fix $n\geq 2$. Let $a_1,a_n\in\mathbb{F}_q[t]$ not all zero. According to the proof of Theorem~\ref{Thm:Reduction_2}, $[a_1]P_1+[a_n]P_n=0$ is equivalent to existence of $g\in k[t]$ such that
        \begin{equation}\label{Eq:Difference_Equation_Carlitz}
            g^{(1)}-(t-\theta)g=a_1P_1+a_nP_n
        \end{equation}
        and the coefficients of $g$ as a polynomial in $k[t]$ lie in $\mathcal{L}(q^{n-2}\cdot(\infty))$ since $\mathrm{ord}_\infty(P_n)=-q^{n-2}$. Note that \eqref{Eq:Difference_Equation_Carlitz} induces a matrix $B_n$ with entries in $\mathbb{F}_q[t]$ and $\deg_t(B_n)\leq 1$, $\rank(B_n)<q^{n-2}+3$. In this case, the upper bound obtained in Theorem~\ref{Thm:Main_Thm} is $q^{n-2}+3$ which is far from the actual degree $n-1$.
        
        We mention that $q^{n-1}+3$ is just a naive upper bound of $\rank(B_n)$ and in general it is still unclear whether the upper bound obtained in Theorem~\ref{Thm:Main_Thm} can be reached or not. We hope to tackle this optimality problem in the near future.
    \end{remark}

\section{Applications and examples}
\subsection{A sufficient condition of linear independence}
    Theorem~\ref{Thm:Main_Thm} gives an effective way to determine the $\mathbb{F}_q[t]$-linear relations among algebraic points on Drinfeld modules. Nevertheless, one may not be able to construct immediately a collection of algebraic points which is $\mathbb{F}_q[t]$-linearly independent on a given Drinfeld module. In what follows, we adopt a different approach to analyze the Frobenius difference equation occurred in Lemma~\ref{Lem:Reduction_1} and we establish a linear independence result for a specific family of Drinfeld modules which enables us to show that the rank of $L$-valued points on this specific family of Drinfeld modules is countably infinite. This is just a special case of Poonen's theorem \cite[Thm.~1]{Poo95}, but our proof avoid the usage of the tameness of the $\mathbb{F}_q[t]$-module $E(L)$ (see \cite[Lem.~4]{Poo95}).
    
    To be more precise, let $L$ be a finite extension of $k$ with $k\subset L\subset \overline{k}$. Let $S\subset M_L$ be a finite collection of places in $L$ which includes all places lying above $\infty\in M_k$. Let 
    $$\mathcal{O}_S:=\{\alpha\in L\mid \mathrm{ord}_v(\alpha)\geq 0~\mbox{for all}~v\not\in S\}\subset L$$
    be the set of $S$-integers in $L$. Then our result is stated as follows.
    
    \begin{theorem}\label{Thm:Linearly_Independence}
        Let $E=(\mathbb{G}_{a/L},\phi)$ be a Drinfeld module $\mathbb{F}_q[t]$-module defined over $L$ with $\phi_t=\theta+\kappa_1\tau+\cdots+\kappa_r\tau^r\in\mathcal{O}_S[\tau]$ and $\kappa_r\in L^\times\cap\overline{\mathbb{F}}_q^\times$. Let $P_1,\dots,P_\ell\in E(L)$. Suppose that
        \begin{enumerate}
            \item $P_1,\dots,P_\ell$ are linearly independent over $\mathbb{F}_q$.
            \item $\mathrm{ord}_v(P_i)>0$ for all $1\leq i\leq\ell$, if $v\in S$.
            \item $\mathrm{ord}_v(P_i)\geq 1-q^r$ for all $1\leq i\leq\ell$, if $v\not\in S$.
        \end{enumerate}
        Then
        \[
            \rank_{\mathbb{F}_q[t]}\mathrm{Span}_{\mathbb{F}_q[t]}\{P_1,\dots,P_\ell\}=\ell.
        \]
        In particular, we have
        \[
            \rank_{\mathbb{F}_q[t]}E(L)=\aleph_0.
        \]
    \end{theorem}
    
    \begin{remark}
        The author thanks Matt Papanikolas for sharing the following proof which is an improvement of a previous weaker result the author obtained.
    \end{remark}
    
    \begin{proof}
        Suppose on the contrary that there exist $a_1,\dots,a_\ell\in\mathbb{F}_q[t]$ not all zero such that
        \[
            \phi_{a_1}(P_1)+\cdots+\phi_{a_\ell}(P_\ell)=0.
        \]
        By Lemma~\ref{Lem:Reduction_1}, there exists $g\in L[t]$ such that
        \[
            a_1P_1+\cdots+a_\ell P_\ell=\kappa_rg^{(r)}+\cdots+\kappa_1g^{(1)}-(t-\theta)g.
        \]
        Let $F:=a_1P_1+\cdots+a_\ell P_\ell\in L[t]$. Since $P_1,\dots,P_\ell\in L$ are linearly independent over $\mathbb{F}_q$, $P_1,\dots,P_\ell\in L[t]$ are linearly independent over $\mathbb{F}_q[t]$ by the linear disjointness of $\mathbb{F}_q[t]$ and $L$. Then $F\neq 0$ in $L[t]$ and hence $g\neq 0$ in $L[t]$.
        
        Let $v\in S$. Then
        \[
            \mathrm{ord}_v(F)=\mathrm{ord}_v(a_1P_1+\cdots a_\ell P_\ell)\geq\min\{\mathrm{ord}_v(P_i)\}_{i=1}^\ell>0.
        \]
        Thus, if we express $F=F_0+F_1t+\cdots+F_mt^m$ where $m\in\mathbb{Z}_{\geq 0}$ and $F_i\in L$ for $0\leq i\leq m$, then $\mathrm{ord}_v(F_i)>0$ for each $0\leq i\leq m$ and $v\in S$. Since $F\neq 0$ in $L[t]$, we may assume that $F_m\neq 0$. Then there exists $w\not\in S$ such that \[
            \mathrm{ord}_w(F)=\min\{\mathrm{ord}_w(F_i)\}_{i=0}^m\leq\mathrm{ord}_w(F_m)<0.
        \]
        Note that we must have $\mathrm{ord}_w(g)<0$, otherwise
        \begin{align*}
            \mathrm{ord}_w(F)&=\mathrm{ord}_w(\kappa_rg^{(r)}+\cdots+\kappa_1g^{(1)}-(t-\theta)g)\\
            &\geq\min\{\mathrm{ord}_w(\kappa_r)+q^r\mathrm{ord}_w(g),\dots,\mathrm{ord}_w(t-\theta)+\mathrm{ord}_w(g)\}\\
            &\geq 0
        \end{align*}
        which leads to a contradiction. 
        
        On the other hand, the fact that $\mathrm{ord}_w(g)<0$ implies that
        \[
            \min\{\mathrm{ord}_w(\kappa_r)+q^r\mathrm{ord}_w(g),\dots,\mathrm{ord}_w(t-\theta)+\mathrm{ord}_w(g)\}=q^r\mathrm{ord}_w(g)
        \]
        and hence $\mathrm{ord}_w(F)=q^r\mathrm{ord}_w(g)$. Now the inequality
        \[
            0>q^r\mathrm{ord}_w(g)=\mathrm{ord}_w(F)=\mathrm{ord}_w(a_1P_1+\cdots+a_\ell P_\ell)\geq 1-q^r
        \]
        leads to a contradiction because of $\mathrm{ord}_w(g)\in\mathbb{Z}$.
        
        Finally, we are going to prove that 
        \[
            \rank_{\mathbb{F}_q[t]}E(L)=\aleph_0.
        \]
        We first write $M_L\setminus S=\{w_i\}_{i\in\mathbb{Z}_{>0}}$. Then for each $j\in\mathbb{Z}_{>0}$ we define divisors of $L$
        \[
            D_j:=\sum_{v\in S}(-1)\cdot v+\sum_{i=1}^j(q^r-1)\cdot w_i.
        \]
        Let $\mathfrak{g}_L$ be the genus of $L$ and $N$ be a positive integer so that for $j>N$
        \[
            \deg(D_j)=(q^r-1)\sum_{i=1}^j\deg(w_i)-\sum_{v\in S}\deg(v)\geq 2\mathfrak{g}_L-1.
        \]
        Then by Riemann-Roch theorem for $j>N$ we have 
        \[
            \dim_{\mathbb{F}_q}\mathcal{L}(D_j)=[\mathbb{F}_L:\mathbb{F}_q]\left(\deg(D_j)-\mathfrak{g}_L+1\right),
        \]
        where $\mathbb{F}_L=L\cap\overline{\mathbb{F}}_q$ is the constant field of $L$. In particular, we have
        \[
            \mathcal{L}(D_{N+1})\subsetneq\mathcal{L}(D_{N+2})\subsetneq\cdots.
        \]
        Then we pick $Q_1\in\mathcal{L}(D_{N+1})$ and for each $i\geq 2$ we pick $Q_i\in\mathcal{L}(D_{N+i})\setminus\mathcal{L}(D_{N+i-1})$. 
        
        Since $\mathrm{Span}_{\mathbb{F}_q[t]}\{Q_i\}_{i\in\mathbb{Z}_{>0}}\subset E(L)$ and any finite non-empty subset of $\{Q_i\}_{i\in\mathbb{Z}_{>0}}$ is a $\mathbb{F}_q[t]$-linearly independent set by the result established above, we conclude that
        \[
            \aleph_0=\rank_{\mathbb{F}_q[t]}\mathrm{Span}_{\mathbb{F}_q[t]}\{Q_i\}_{i\in\mathbb{Z}_{>0}}\leq\rank_{\mathbb{F}_q[t]}E(L)\leq\aleph_0
        \]
        and hence
        \[
            \rank_{\mathbb{F}_q[t]}E(L)=\aleph_0.
        \]
    \end{proof}
    
    The following example gives an explicit construction of $\mathbb{F}_q[t]$-linearly independent set of arbitrary cardinality in the case $L=k$.
    
    \begin{example}\label{Ex:Linearly_Independent}
        Let $S\subset M_k$ be a finite collection of places in $k$ which includes $\infty$. Let
        \[
            A_S:=\{\alpha\in k\mid\mathrm{ord}_v(\alpha)\geq 0~\mbox{for all}~v\not\in S~\}\subset k.
        \]
        Let $E=(\mathbb{G}_{a/L},\phi)$ be a Drinfeld module $\mathbb{F}_q[t]$-module defined over $k$ with
        \[
            \phi_t=\theta+\kappa_1\tau+\cdots+\kappa_r\tau^r\in A_S[\tau]
        \]
        and $\kappa_r\in \mathbb{F}_q^\times$. For each $v\in M_k$, we set $f_v$ to be the monic irreducible polynomial corresponding to the place $v$. Let $f:=\prod_{v\in S\setminus\{\infty\}}f_v$ and $N:=\deg_\theta(f)$. Let $S'$ be the collection of finite places whose degree is strictly greater than $N$. Then for each $v\in S'$, we may define $P_v:=f/f_v$. It is clear that $S'$ is a countably infinite set and any finite non-empty subset of $\{P_v\}_{v\in S'}$ is $\mathbb{F}_q[t]$-linearly independent by Theorem~\ref{Thm:Linearly_Independence}. Thus, we conclude that
        \[
            \rank_{\mathbb{F}_q[t]}E(k)=\aleph_0.
        \]
    \end{example}
    
\subsection{Algebraic independence of Drinfeld logarithms at algebraic points}
    Let $E$ be a Drinfeld $\mathbb{F}_q[t]$-module of rank $r$ defined over $\mathbb{C}_\infty$. We can associate a discrete free $A$-module $\Lambda_E\subset\mathbb{C}_\infty$ of rank $r$ for $E$. Any element in $\Lambda_E$ is called a \emph{period} of $E$. For the convenience of later use, we express $\Lambda_E=A\lambda_1+\cdots A\lambda_r$ for some $\lambda_1,\dots,\lambda_r\in\Lambda_E$. 
    
    The \emph{exponential function} of $E$ is defined by the power series 
    \[
        \exp_E(z):=z\prod_{0\neq\lambda\in\Lambda_E}(1-z/\lambda).
    \]
    Note that $\exp_E(z)$ induces a surjective, entire $\mathbb{F}_q$-linear function on $\mathbb{C}_\infty$, from which we have the following analytic uniformization
    \begin{equation}\label{Eq:Drinfeld_Modules_Uniformization}
        0\to \Lambda_E\hookrightarrow\mathbb{C}_\infty\overset{\exp_\Lambda(\cdot)}{\twoheadrightarrow}E(\mathbb{C}_\infty)\to 0.
    \end{equation}
    It satisfies that $\exp_E(a(\theta)z)=\phi_a(\exp(z))$ for all $a\in\mathbb{F}_q[t]$. The \emph{logarithmic function} $\log_E$ of the Drinfeld $\mathbb{F}_q[t]$-module $E$ is defined by the formal inverse of $\exp_E$. It satisfies that $a(\theta)\log_E(z)=\log_E(\phi_a(z))$ for all $a\in\mathbb{F}_q[t]$. In general, $\log_E$ converges on a disk with center at $0$ of sufficiently small non-zero radius in $\mathbb{C}_\infty$ even though $\exp_E$ defines an entire function on $\mathbb{C}_\infty$.
    
    Recall that $K_E=\End(E)\otimes_{\mathbb{F}_q[t]}\mathbb{F}_q(t)$ is a finite extension over $\mathbb{F}_q(t)$ of degree $s$, which can be identified as a subfield of $\mathbb{C}_\infty$ via $\partial:K_E\to\mathbb{C}_\infty$ defined in the introduction. We denote by $\mathcal{K}_E:=\partial(K_E)\subset\mathbb{C}_\infty$. If we denote by $k\Lambda_E:=k\lambda_1+\cdots k\lambda_r$, then one checks directly that $\mathcal{K}_E(k\Lambda_E)\subset k\Lambda_E$. In particular, we have
    \[
        k\Lambda_E=\mathrm{Span}_{\mathcal{K}_E}\{\lambda_1,\dots,\lambda_r\}.
    \]
    It follows that
    \[
        \dim_{\mathcal{K}_E}\mathrm{Span}_{\mathcal{K}_E}\{\lambda_1,\dots,\lambda_r\}=\dim_{k}(k\Lambda_E)/[K_E:\mathbb{F}_q(t)]=r/s.
    \]
    
    As an application of Theorem~\ref{Thm:Main_Thm}, we have the following result, which provides a practical way to determine all the algebraic relations among Drinfeld logarithms at algebraic points.
    
    \begin{theorem}\label{Thm:Algebraic_Independence}
        Let $L/k$ be a finite extension and $E$ be a Drinfeld $\mathbb{F}_q[t]$-module defined over $L$ with $\End(E)=\mathbb{F}_q[t]u_1+\cdots\mathbb{F}_q[t]u_s$. Let $Q_1,\dots,Q_\ell\in\mathbb{C}_\infty$ so that $P_1:=\exp_E(Q_1),\dots,P_\ell:=\exp_E(Q_\ell)$ are distinct non-zero elements in $E(L)$. Then we have
        \[
            \dim_{\mathcal{K}_E}\mathrm{Span}_{\mathcal{K}_E}\{\lambda_1,\dots,\lambda_r,Q_1,\dots,Q_\ell\}=r/s+\rank_{\End(E)}\mathrm{Span}_{\End(E)}\{P_1,\dots,P_\ell\}.
        \]
        In particular, if we set and
        $$D:=\left(-\mathrm{div}(u_1(P_1),\dots,u_1(P_\ell),\dots,u_s(P_1),\dots,u_s(P_\ell))\right)\vee D_E$$
        and $d:=\dim_{\mathbb{F}_q}\mathcal{L}(D)$, then $Q_1,\dots,Q_\ell$ are algebraically independent over $\ok$ whenever
        \[
            \underset{1\leq j\leq\ell}{\sum_{1\leq i\leq s}}\phi_{a_{ij}}(u_i(P_j))\neq 0
        \]
        for all $a_{ij}\in\mathbb{F}_q[t]$ and $\deg_t(a_{ij})<d+s\ell$.
    \end{theorem}
    
    \begin{remark}
        Note that the cardinality of the set
        \[
            |\{\underset{1\leq j\leq\ell}{\sum_{1\leq i\leq s}}\phi_{a_{ij}}(u_i(P_j))\mid a_{ij}\in\mathbb{F}_q[t],~\deg_t(a_{ij})<d+s\ell,~1\leq i\leq s,~1\leq j\leq\ell\}|=s\ell q^{d+s\ell}
        \]
        is finite and computable. Thus, we provide an effective result to decide whether Drinfeld logarithms at algebraic points are algebraically independent.
    \end{remark}
    
    \begin{proof}
        Let $C_1,\dots,C_\ell\in\End(E)$, not all zero, so that
        \[
            C_1(P_1)+\cdots+C_\ell(P_\ell)=0.
        \]
        Then by the functorial property of $\exp_E$ \cite[Thm.~3]{And86} we have
        \[
            \exp_E((\partial C_1)Q_1+\cdots+(\partial C_\ell)Q_\ell)=0.
        \]
        In other words, $(\partial C_1)Q_1+\cdots(\partial C_\ell)Q_\ell\in\Lambda_E$, and hence
        \[
            (\partial C_1)Q_1+\cdots(\partial C_\ell)Q_\ell\in\mathrm{Span}_{\mathcal{K}_E}\{\lambda_1,\dots,\lambda_r\}.
        \]
        Since $(\partial C_1),\dots,(\partial C_\ell)\in\mathcal{K}_E$ and $\dim_{\mathcal{K}_E}\mathrm{Span}_{\mathcal{K}_E}\{\lambda_1,\dots,\lambda_r\}=r/s$, every non-trivial $\End(E)$-linear relation among $\{P_1,\dots,P_\ell\}$ lifts to a non-trivial $\mathcal{K}_E$-linear relation among $\{\lambda_1,\dots,\lambda_r,Q_1,\dots,Q_\ell\}$ that is not a non-trivial $\mathcal{K}_E$-linear relation among $\{\lambda_1,\dots,\lambda_r\}$. Thus, we conclude that
        \[
            \dim_{\mathcal{K}_E}\mathrm{Span}_{\mathcal{K}_E}\{\lambda_1,\dots,\lambda_r,Q_1,\dots,Q_\ell\}\leq r/s+\rank_{\End(E)}\mathrm{Span}_{\End(E)}\{P_1,\dots,P_\ell\}.
        \]
        
        On the other hand, let $\alpha_1,\dots,\alpha_r\in\mathcal{K}_E$ and $\beta_1,\dots,\beta_\ell\in\mathcal{K}_E$, not all zero, so that
        \[
            \sum_{i=1}^r\alpha_i\lambda_i+\sum_{j=1}^\ell\beta_jQ_j=0.
        \]
        There exists $\gamma\in A$ so that $\gamma\alpha_i\in\partial(\End(E))$ and $\gamma\beta_j\in\partial(\End(E))$. Then we have
        \[
            0=\exp_E(\sum_{i=1}^r\gamma\alpha_i\lambda_i+\sum_{j=1}^\ell\gamma\beta_jQ_j)=C_1(P_1)+\cdots+C_\ell(P_\ell)
        \]
        for some $C_1,\dots,C_\ell\in\End(E)$, not all zero, and $\partial C_j=\gamma\beta_j$. Thus, every non-trivial $\mathcal{K}_E$-linear relation among $\{\lambda_1,\dots,\lambda_r,Q_1,\dots,Q_\ell\}$ that is not a non-trivial $\mathcal{K}_E$-linear relation among $\{\lambda_1,\dots,\lambda_r\}$ lifts to a non-trivial $\End(E)$-linear relation among $\{P_1,\dots,P_\ell\}$. Therefore,
        \[
            \dim_{\mathcal{K}_E}\mathrm{Span}_{\mathcal{K}_E}\{\lambda_1,\dots,\lambda_r,Q_1,\dots,Q_\ell\}\geq r/s+\rank_{\End(E)}\mathrm{Span}_{\End(E)}\{P_1,\dots,P_\ell\},
        \]
        and the desired equality now follows.
        
        Finally, since $\End(E)$-linear relations among $\{P_1,\dots,P_\ell\}$ are those coming from $\mathbb{F}_q[t]$-linear relations among $\{u_1(P_1),\dots,u_1(P_\ell),\dots,u_s(P_1),\dots,u_s(P_\ell)\}$ and vice versa, Theorem~\ref{Thm:Main_Thm} shows that 
        \[
            \rank_{\End(E)}\mathrm{Span}_{\End(E)}\{P_1,\dots,P_\ell\}=\ell
        \]
        whenever
        \[
            \underset{1\leq j\leq\ell}{\sum_{1\leq i\leq s}}\phi_{a_{ij}}(u_i(P_j))\neq 0
        \]
        for all $a_{ij}\in\mathbb{F}_q[t]$ and $\deg_t(a_{ij})<d+s\ell$. In this case, we have
        \[
            \dim_{\mathcal{K}_E}\mathrm{Span}_{\mathcal{K}_E}\{Q_1,\dots,Q_\ell\}=\ell
        \]
        and the desired algebraic independence of $Q_1,\dots,Q_\ell$ follows immediately from \cite[Thm.~1.1.1]{CP12}.
    \end{proof}

    \begin{example}
        Let $E=(\mathbb{G}_{a/k},\phi)$ be a rank $2$ Drinfeld $\mathbb{F}_q[t]$-module defined over $k$ with $\phi_t=\theta+(1/\theta)\tau+\tau^2$. One computes directly that the $j$-invariant of $E$ is
        $$j(E)=\left(\frac{1}{\theta}\right)^{q+1}\in k\setminus A.$$
    
        Since $A$ is integrally closed in $k$, $j(E)$ is not integral over $A$ and it implies that $E$ is a Drinfeld $\mathbb{F}_q[t]$-module without complex multiplication since the $j$-invariant of Drinfeld $\mathbb{F}_q[t]$-modules with complex multiplication must be an algebraic integer \cite{Gek83} (see \cite{Ham03} for higher rank case). Now we consider the case $q=2$ and two points $P_1=\theta,~P_2=\theta+1$ on the Drinfeld $\mathbb{F}_q[t]$-module $E$. By \cite[Cor.~4.2]{EGP13}, one checks directly that $\log_E$ converges at $P_1$ and $P_2$.
    
        We aim to show that $\log_E(\theta)$ and $\log_E(\theta+1)$ are algebraically independent over $\overline{k}$. Following the same notation given in Theorem~\ref{Thm:Main_Thm} and the proof of Theorem~\ref{Thm:Reduction_2}, one checks directly that
        $$\mathrm{div}(P_1,P_2)=(-1)\cdot(\infty),~\mbox{and}~D_E=1\cdot(\theta)+1\cdot(\infty).$$ 
        Hence we get the divisor 
        $$D=\left(-\mathrm{div}(P_1,P_2)\right)\vee D_E=1\cdot(\infty)+1\cdot(\theta).$$
        Then one may use Riemann-Roch theorem to conclude that $\dim_{\mathbb{F}_q}\mathcal{L}(D)=3$. In fact, one has 
        $$\mathcal{L}(D)=\mathbb{F}_q+\mathbb{F}_q\cdot\theta+\mathbb{F}_q\cdot\frac{1}{\theta}.$$
        Thus by applying Theorem~\ref{Thm:Main_Thm}, we know that if $P_1$ and $P_2$ are $\mathbb{F}_q[t]$-linearly dependent, then we can find generators of their $\mathbb{F}_q[t]$-linear relations with degree less than or equal to
        $d+\ell=\dim_{\mathbb{F}_q}\mathcal{L}(D)+2=5$. Then one checks by using SageMath that the quantity $\phi_{a_1}(P_1)+\phi_{a_2}(P_2)$ is non-vanishing for all $a_i\in\mathbb{F}_q[t]$, not all zero, with $\deg_t(a_i)<5$, $i=1,2$. Consequently, we deduce that $\theta,~\theta+1$ are $\mathbb{F}_q[t]$-linearly independent in $E(k)$ and hence $\log_E(\theta),~\log_E(\theta+1)$ are linearly independent over $k$. Note that the field of fractions of $\End(\phi)=A$ coincides with $k$. Finally, we apply \cite[Thm.~1.1.1]{CP12} to conclude that $\log_E(\theta)$ and $\log_E(\theta+1)$ are algebraically independent over $\overline{k}$ as desired.
    \end{example}

\bibliographystyle{alpha}

\end{document}